\documentclass[12pt,a4paper]{amsart}
\usepackage{amssymb,amsxtra}
\usepackage[english,french]{babel}

\newtheorem{thm}{Theorem}[section]
\newtheorem{lem}[thm]{Lemma}

\theoremstyle{definition}

\theoremstyle{remark}

\theoremstyle{plain}

\theoremstyle{remark}

\newtheorem*{example}{Example}
\numberwithin{equation}{section}

\begin{document}

\title{ Enumeration of some particular   n-times  Persymmetric  Matrices over $\mathbb{F}_{2} $ by rank}
\author{Jorgen~Cherly}
\address{D\'epartement de Math\'ematiques, Universit\'e de
    Brest, 29238 Brest cedex~3, France}
\email{Jorgen.Cherly@univ-brest.fr}
\email{andersen69@wanadoo.fr}

\maketitle 
\begin{abstract}
Dans cet article nous comptons le nombre de certaines  n-fois matrices persym\' etriques de rang i sur $ \mathbb {F} _ {2} . $

 \end{abstract}

\selectlanguage{english}

\begin{abstract}
In this paper we count the number of some particular  n-times persymmetric rank i matrices over  $ \mathbb{F}_{2}.$
 \end{abstract}
 
 \maketitle 
\newpage
\tableofcontents
\newpage
   \section{Some notations concerning the field of Laurent Series $ \mathbb{F}_{2}((T^{-1})) $ }
  \label{sec 1}  
  We denote by $ \mathbb{F}_{2}\big(\big({T^{-1}}\big) \big)
 = \mathbb{K} $ the completion
 of the field $\mathbb{F}_{2}(T), $  the field of  rational fonctions over the
 finite field\; $\mathbb{F}_{2}$,\; for the  infinity  valuation \;
 $ \mathfrak{v}=\mathfrak{v}_{\infty }$ \;defined by \;
 $ \mathfrak{v}\big(\frac{A}{B}\big) = degB -degA $ \;
 for each pair (A,B) of non-zero polynomials.
 Then every element non-zero t in
  $\mathbb{F}_{2}\big(\big({\frac{1}{T}}\big) \big) $
 can be expanded in a unique way in a convergent Laurent series
                              $  t = \sum_{j= -\infty }^{-\mathfrak{v}(t)}t_{j}T^j
                                 \; where\; t_{j}\in \mathbb{F}_{2}. $\\
  We associate to the infinity valuation\; $\mathfrak{v}= \mathfrak{v}_{\infty }$
   the absolute value \; $\vert \cdot \vert_{\infty} $\; defined by \;
  \begin{equation*}
  \vert t \vert_{\infty} =  \vert t \vert = 2^{-\mathfrak{v}(t)}. \\
\end{equation*}
    We denote  E the  Character of the additive locally compact group
$  \mathbb{F}_{2}\big(\big({\frac{1}{T}}\big) \big) $ defined by \\
\begin{equation*}
 E\big( \sum_{j= -\infty }^{-\mathfrak{v}(t)}t_{j}T^j\big)= \begin{cases}
 1 & \text{if      }   t_{-1}= 0, \\
  -1 & \text{if      }   t_{-1}= 1.
    \end{cases}
\end{equation*}
  We denote $\mathbb{P}$ the valuation ideal in $ \mathbb{K},$ also denoted the unit interval of  $\mathbb{K},$ i.e.
  the open ball of radius 1 about 0 or, alternatively, the set of all Laurent series 
   $$ \sum_{i\geq 1}\alpha _{i}T^{-i}\quad (\alpha _{i}\in  \mathbb{F}_{2} ) $$ and, for every rational
    integer j,  we denote by $\mathbb{P}_{j} $
     the  ideal $\left\{t \in \mathbb{K}|\; \mathfrak{v}(t) > j \right\}. $
     The sets\; $ \mathbb{P}_{j}$\; are compact subgroups  of the additive
     locally compact group \; $ \mathbb{K}. $\\
      All $ t \in \mathbb{F}_{2}\Big(\Big(\frac{1}{T}\Big)\Big) $ may be written in a unique way as
$ t = [t] + \left\{t\right\}, $ \;  $  [t] \in \mathbb{F}_{2}[T] ,
 \; \left\{t\right\}\in \mathbb{P}  ( =\mathbb{P}_{0}). $\\
 We denote by dt the Haar measure on \; $ \mathbb{K} $\; chosen so that \\
  $$ \int_{\mathbb{P}}dt = 1. $$\\
  
  $$ Let \quad
  (t_{1},t_{2},\ldots,t_{n} )
 =  \big( \sum_{j=-\infty}^{-\nu(t_{1})}\alpha _{j}^{(1)}T^{j},  \sum_{j=-\infty}^{-\nu(t_{2})}\alpha _{j}^{(2)}T^{j} ,\ldots, \sum_{j=-\infty}^{-\nu(t_{n})}\alpha _{j}^{(n)}T^{j}\big) \in  \mathbb{K}^{n}. $$ 
 We denote $\psi  $  the  Character on  $(\mathbb{K}^n, +) $ defined by \\
 \begin{align*}
  \psi \big( \sum_{j=-\infty}^{-\nu(t_{1})}\alpha _{j}^{(1)}T^{j},  \sum_{j=-\infty}^{-\nu(t_{2})}\alpha _{j}^{(2)}T^{j} ,\ldots, \sum_{j=-\infty}^{-\nu(t_{n})}\alpha _{j}^{(n)}T^{j}\big) & = E \big( \sum_{j=-\infty}^{-\nu(t_{1})}\alpha _{j}^{(1)}T^{j}\big) \cdot E\big( \sum_{j=-\infty}^{-\nu(t_{2})}\alpha _{j}^{(2)}T^{j}\big)\cdots E\big(  \sum_{j=-\infty}^{-\nu(t_{n})}\alpha _{j}^{(n)}T^{j}\big) \\
  & = 
    \begin{cases}
 1 & \text{if      }     \alpha _{-1}^{(1)} +    \alpha _{-1}^{(2)}  + \ldots +   \alpha _{-1}^{(n)}   = 0 \\
  -1 & \text{if      }    \alpha _{-1}^{(1)} +    \alpha _{-1}^{(2)}  + \ldots +   \alpha _{-1}^{(n)}   =1                                                                                                                          
    \end{cases}
  \end{align*}

   \section{Some results concerning  n-times persymmetric matrices over  $ \mathbb{F}_{2}$}
  \label{sec 2}  
     $$ Set\quad
  (t_{1},t_{2},\ldots,t_{n} )
 =  \big( \sum_{i\geq 1}\alpha _{i}^{(1)}T^{-i}, \sum_{i \geq 1}\alpha  _{i}^{(2)}T^{-i},\sum_{i \geq 1}\alpha _{i}^{(3)}T^{-i},\ldots,\sum_{i \geq 1}\alpha _{i}^{(n)}T^{-i}   \big) \in  \mathbb{P}^{n}. $$

     Denote by $D^{\left[2 \atop{\vdots \atop 2}\right]\times k}(t_{1},t_{2},\ldots,t_{n} ) $
    
    the following $2n \times k $ \;  n-times  persymmetric  matrix  over the finite field  $\mathbb{F}_{2} $ 
    
  \begin{displaymath}
   \left (  \begin{array} {cccccccc}
\alpha  _{1}^{(1)} & \alpha  _{2}^{(1)}  &   \alpha_{3}^{(1)} &   \alpha_{4}^{(1)} &   \alpha_{5}^{(1)} &  \alpha_{6}^{(1)}  & \ldots  &  \alpha_{k}^{(1)} \\
\alpha  _{2}^{(1)} & \alpha  _{3}^{(1)}  &   \alpha_{4}^{(1)} &   \alpha_{5}^{(1)} &   \alpha_{6}^{(1)} &  \alpha_{7}^{(1)} & \ldots  &  \alpha_{k+1}^{(1)} \\ 
\hline \\
\alpha  _{1}^{(2)} & \alpha  _{2}^{(2)}  &   \alpha_{3}^{(2)} &   \alpha_{4}^{(2)} &   \alpha_{5}^{(2)} &  \alpha_{6}^{(2)} & \ldots   &  \alpha_{k}^{(2)} \\
\alpha  _{2}^{(2)} & \alpha  _{3}^{(2)}  &   \alpha_{4}^{(2)} &   \alpha_{5}^{(2)}&   \alpha_{6}^{(2)} &  \alpha_{7}^{(2)}  & \ldots  &  \alpha_{k+1}^{(2)} \\ 
\hline\\
\alpha  _{1}^{(3)} & \alpha  _{2}^{(3)}  &   \alpha_{3}^{(3)}  &   \alpha_{4}^{(3)} &   \alpha_{5}^{(3)} &  \alpha_{6}^{(3)} & \ldots  &  \alpha_{k}^{(3)} \\
\alpha  _{2}^{(3)} & \alpha  _{3}^{(3)}  &   \alpha_{4}^{(3)}&   \alpha_{5}^{(3)} &   \alpha_{6}^{(3)}  &  \alpha_{7}^{(3)} & \ldots  &  \alpha_{k+1}^{(3)} \\ 
\hline \\
\vdots & \vdots & \vdots  & \vdots  & \vdots & \vdots  & \vdots & \vdots \\
\hline \\
\alpha  _{1}^{(n)} & \alpha  _{2}^{(n)}  &   \alpha_{3}^{(n)} &   \alpha_{4}^{(n)} &   \alpha_{5}^{(n)}  &  \alpha_{6}^{(n)} & \ldots  &  \alpha_{k}^{(n)} \\
\alpha  _{2}^{(n)} & \alpha  _{3}^{(n)}  &   \alpha_{4}^{(n)}&   \alpha_{5}^{(n)} &   \alpha_{6}^{(n)}  &  \alpha_{7}^{(n)} & \ldots  &  \alpha_{k+1}^{(n)} \\ 
\end{array} \right )  
\end{displaymath} 
We denote by  $ \Gamma_{i}^{\left[2\atop{\vdots \atop 2}\right]\times k}$  the number of rank i n-times persymmetric matrices over $\mathbb{F}_{2}$ of the above form :  \\

  Let $ \displaystyle  f (t_{1},t_{2},\ldots,t_{n} ) $  be the exponential sum  in $ \mathbb{P}^{n} $ defined by\\
    $(t_{1},t_{2},\ldots,t_{n} ) \displaystyle\in \mathbb{P}^{n}\longrightarrow \\
    \sum_{deg Y\leq k-1}\sum_{deg U_{1}\leq  1}E(t_{1} YU_{1})
  \sum_{deg U_{2} \leq 1}E(t _{2} YU_{2}) \ldots \sum_{deg U_{n} \leq 1} E(t _{n} YU_{n}). $\vspace{0.5 cm}\\
    Then
  $$     f_{k} (t_{1},t_{2},\ldots,t_{n} ) =
  2^{2n+k- rank\big[ D^{\left[2\atop{\vdots \atop 2}\right]\times k}(t_{1},t_{2},\ldots,t_{n} )\big] } $$

    Hence  the number denoted by $ R_{q}^{(k)} $ of solutions \\
  
 $(Y_1,U_{1}^{(1)},U_{2}^{(1)}, \ldots,U_{n}^{(1)}, Y_2,U_{1}^{(2)},U_{2}^{(2)}, 
\ldots,U_{n}^{(2)},\ldots  Y_q,U_{1}^{(q)},U_{2}^{(q)}, \ldots,U_{n}^{(q)}   ) \in (\mathbb{F}_{2}[T])^{(n+1)q}$ \vspace{0.5 cm}\\
 of the polynomial equations  \vspace{0.5 cm}
  \[\left\{\begin{array}{c}
 Y_{1}U_{1}^{(1)} + Y_{2}U_{1}^{(2)} + \ldots  + Y_{q}U_{1}^{(q)} = 0  \\
    Y_{1}U_{2}^{(1)} + Y_{2}U_{2}^{(2)} + \ldots  + Y_{q}U_{2}^{(q)} = 0\\
    \vdots \\
   Y_{1}U_{n}^{(1)} + Y_{2}U_{n}^{(2)} + \ldots  + Y_{q}U_{n}^{(q)} = 0 
 \end{array}\right.\]
 
    $ \Leftrightarrow
    \begin{pmatrix}
   U_{1}^{(1)} & U_{1}^{(2)} & \ldots  & U_{1}^{(q)} \\ 
      U_{2}^{(1)} & U_{2}^{(2)}  & \ldots  & U_{2}^{(q)}  \\
\vdots &   \vdots & \vdots &   \vdots   \\
U_{n}^{(1)} & U_{n}^{(2)}   & \ldots  & U_{n}^{(q)} \\
 \end{pmatrix}  \begin{pmatrix}
   Y_{1} \\
   Y_{2}\\
   \vdots \\
   Y_{q} \\
  \end{pmatrix} =   \begin{pmatrix}
  0 \\
  0 \\
  \vdots \\
  0 
  \end{pmatrix} $\\

    satisfying the degree conditions \\
                   $$  degY_i \leq k-1 ,
                   \quad degU_{j}^{(i)} \leq 1, \quad  for \quad 1\leq j\leq n  \quad 1\leq i \leq q $$ \\
  is equal to the following integral over the unit interval in $ \mathbb{K}^{n} $
    $$ \int_{\mathbb{P}^{n}} f_{k}^{q}(t_{1},t_{2},\ldots,t_{n}) dt_{1}dt _{2}\ldots dt _{n}. $$
  Observing that $ f (t_{1},t_{2},\ldots,t_{n} ) $ is constant on cosets of $ \prod_{j=1}^{n}\mathbb{P}_{k+1} $ in $ \mathbb{P}^{n} $\;
  the above integral is equal to 
  
  \begin{equation}
  \label{eq 2.1}
 2^{q(2n+k) - (k+1)n}\sum_{i = 0}^{k}
  \Gamma_{i}^{\left[2\atop{\vdots \atop 2}\right]\times k} 2^{-iq} =  R_{q}^{(k)} \quad \text{where} \; k\leqslant 2n
 \end{equation}
 
 \begin{eqnarray}
 \label{eq 2.2}
\text{ Recall that $ R_{q}^{(k)}$ is equal to the number of solutions of the polynomial system} \nonumber \\
    \begin{pmatrix}
   U_{1}^{(1)} & U_{1}^{(2)} & \ldots  & U_{1}^{(q)} \\ 
      U_{2}^{(1)} & U_{2}^{(2)}  & \ldots  & U_{2}^{(q)}  \\
\vdots &   \vdots & \vdots &   \vdots   \\
U_{n}^{(1)} & U_{n}^{(2)}   & \ldots  & U_{n}^{(q)} \\
 \end{pmatrix}  \begin{pmatrix}
   Y_{1} \\
   Y_{2}\\
   \vdots \\
   Y_{q} \\
  \end{pmatrix} =   \begin{pmatrix}
  0 \\
  0 \\
  \vdots \\
  0 
  \end{pmatrix} \\
 \text{ satisfying the degree conditions}\nonumber \\
                     degY_i \leq k-1 ,
                   \quad degU_{j}^{(i)} \leq 1, \quad  for \quad 1\leq j\leq n  \quad 1\leq i \leq q  \nonumber
 \end{eqnarray}

 From \eqref{eq 2.1} we obtain for q = 1\\
   \begin{align}
  \label{eq 2.3}
 2^{k-(k-1)n}\sum_{i = 0}^{k}
 \Gamma_{i}^{\left[2\atop{\vdots \atop 2}\right]\times k} 2^{-i} =  R_{1}^{(k)} = 2^{2n}+2^k-1
  \end{align}

We have obviously \\

   \begin{align}
  \label{eq 2.4}
 \sum_{i = 0}^{k}
 \Gamma_{i}^{\left[2\atop{\vdots \atop 2}\right]\times k}  = 2^{(k+1)n}  
 \end{align}

From  the fact that the number of rank one persymmetric  matrices over $\mathbb{F}_{2}$ is equal to three  we obtain using
 combinatorial methods  : \\
 
    \begin{equation}
  \label{eq 2.5}
 \Gamma_{1}^{\left[2\atop{\vdots \atop 2}\right]\times k}  = (2^{n}-1)\cdot 3
  \end{equation}
  For more details see Cherly  [9],[10]
     \section{Enumeration of some special matrices over $\mathbb{F}_{2}[T] $ by rank}
  \label{sec 3}  
 Denote by  $ r _{i}^{n\times 2} =  r _{i}$  for  $ i\in \{0,1,2\} $  the number of $n\times 2 $  rank i matrices over $  \mathbb{F}_{2}[T]  $
of the form : \\[0,2 cm]
 $ M= \begin{pmatrix}
   U_{1}^{(1)} & U_{1}^{(2)}\\    U_{2}^{(1)} & U_{2}^{(2)}\\
\vdots &   \vdots        \\
U_{n}^{(1)} & U_{n}^{(2)}\\
 \end{pmatrix} =  \begin{pmatrix}
  a_{1}+b_{1}T & c_{1}+d_{1}T \\  
    a_{2}+b_{2}T & c_{2}+d_{2}T \\
\vdots &   \vdots        \\
  a_{n}+b_{n}T & c_{n}+d_{n}T\\
  \end{pmatrix} ,$\\
   where $ deg U_{j}^{(i)} \leq 1 \quad  for \quad 1\leq j\leq n  \quad 1\leq i \leq 2 $ 
   that is  $  a_{j}, b_{j}, c_{j}, d_{j} \in  \mathbb{F}_{2} \; \text{for} \; 1\leqslant j \leqslant n$\\[0.2 cm]
 Obviously we have : $ r_{0} =1$  and  $ r_{0} + r_{1}+ r_{2} = 4^{2n}$ \\
  \textbf{To compute  $ r_{1} $ we proceed as follows }:\\

 The number of $n\times 2 $  rank 1 matrices over $  \mathbb{F}_{2}[T]  $
of the form : \\[0,2 cm]
  $ M =   \begin{pmatrix}
  0 & c_{1}+d_{1}T \\  
    0 & c_{2}+d_{2}T \\
\vdots &   \vdots        \\
  0 & c_{n}+d_{n}T\\
  \end{pmatrix} ,$
where    $   c_{j}, d_{j} \in  \mathbb{F}_{2} \; \text{for} \; 1\leqslant j \leqslant n $\\[0.2 cm]
is obviously  equal to $4^n-1$,\\[0.2 cm]
just as the number of $n\times 2 $  rank 1 matrices over $  \mathbb{F}_{2}[T]  $
of the form : \\[0,2 cm]
  $ M =   \begin{pmatrix}
  a_{1}+b_{1}T & 0 \\  
    a_{2}+b_{2}T & 0 \\
\vdots &   \vdots        \\
  a_{n}+b_{n}T & 0\\
  \end{pmatrix} ,$
where    $  a_{j}, b_{j} \in  \mathbb{F}_{2} \; \text{for} \; 1\leqslant j \leqslant n$\\[0.2 cm]
\\[0.2 cm]
and  the number of $n\times 2 $  rank 1 matrices over $  \mathbb{F}_{2}[T]  $
of the form : \\[0,2 cm]
  $ M =   \begin{pmatrix}
  a_{1}+b_{1}T & a_{1}+b_{1}T \\  
    a_{2}+b_{2}T & a_{2}+b_{2}T \\
\vdots &   \vdots        \\
  a_{n}+b_{n}T & a_{n}+b_{n}T\\
  \end{pmatrix} ,$
where    $  a_{j}, b_{j}\in  \mathbb{F}_{2} \; \text{for} \; 1\leqslant j \leqslant n$\\[0.2 cm]

Let M be one of the following  $n\times 2 $  rank 1 matrices over $  \mathbb{F}_{2}[T]  $ \\
\begin{align}
\label{eq 3.1}
    \begin{pmatrix}
  0 & c_{1}+d_{1}T \\  
    0 & c_{2}+d_{2}T \\
\vdots &   \vdots        \\
  0 & c_{n}+d_{n}T\\
  \end{pmatrix} ,
      \begin{pmatrix}
  a_{1}+b_{1}T & 0 \\  
    a_{2}+b_{2}T & 0 \\
\vdots &   \vdots        \\
  a_{n}+b_{n}T & 0\\
  \end{pmatrix} ,
     \begin{pmatrix}
  a_{1}+b_{1}T & a_{1}+b_{1}T \\  
    a_{2}+b_{2}T & a_{2}+b_{2}T \\
\vdots &   \vdots        \\
  a_{n}+b_{n}T & a_{n}+b_{n}T\\
  \end{pmatrix} ,
\end{align}

Thus the number of  $n\times 2 $  rank 1 matrices over $  \mathbb{F}_{2}[T]  $ of the form \eqref{eq 3.1} is equal to $3\cdot(2^{2n}-1)$

To  obtain  $ r_{1} $ it remains only to compute 
  the number of $n\times 2 $  matrices over $  \mathbb{F}_{2}[T]  $
of the form : \\[0,2 cm]
\begin{align}
  \label{eq 3.2}
    M =   \begin{pmatrix}
  a_{1}+b_{1}T & c_{1}+d_{1}T \\  
    a_{2}+b_{2}T & c_{2}+d_{2}T \\
\vdots &   \vdots        \\
  a_{n}+b_{n}T & c_{n}+d_{n}T\\
  \end{pmatrix} 
\text{where}  \begin{pmatrix}
  a_{1}+b_{1}T\\  
    a_{2}+b_{2}T  \\
\vdots      \\
  a_{n}+b_{n}T 
  \end{pmatrix} \neq  \begin{pmatrix}
 c_{1}+d_{1}T \\  
   c_{2}+d_{2}T \\
\vdots    \\
  c_{n}+d_{n}T\\
  \end{pmatrix} \\
  \text{and} \quad  rank M = rank  \begin{pmatrix}
  a_{1}+b_{1}T\\  
    a_{2}+b_{2}T  \\
\vdots      \\
  a_{n}+b_{n}T 
  \end{pmatrix} = rank  \begin{pmatrix}
 c_{1}+d_{1}T \\  
   c_{2}+d_{2}T \\
\vdots    \\
  c_{n}+d_{n}T\\
  \end{pmatrix} = 1 \nonumber
\end{align}

  Since the rank of the matrix 
$ \begin{pmatrix}
  a_{1}+b_{1}T\\  
    a_{2}+b_{2}T  \\
\vdots      \\
  a_{n}+b_{n}T 
  \end{pmatrix}$ is equal to one we can assume 
  that there exists $i\in[1,n] $ such that  $ a_{i}+b_{i}T \in \{1,T,1+T\} .$\\
  Let for instance  $ a_{i}+b_{i}T $ be equal to T.\\
  Then $\begin{vmatrix}
    T & c_{i}+d_{i}T \\  
    a_{j}+b_{j}T & c_{j}+d_{j}T 
\end{vmatrix} =  T( c_{j}+d_{j}T)+( a_{j}+b_{j}T)( c_{i}+d_{i}T)$ is equal to 0 for every $ j \neq i$ since rank M is equal to one.\\
If $ c_{i}+d_{i}T$  is equal to 0, we conclude that $ c_{j}+d_{j}T$ is equal to zero for every j,which contradict the fact that the rank of the matrix 
$   \begin{pmatrix}
 c_{1}+d_{1}T \\  
   c_{2}+d_{2}T \\
\vdots    \\
  c_{n}+d_{n}T\\
  \end{pmatrix}  $ is equal to one.\\
  We then deduce that  $ c_{i}+d_{i}T \in \{1,T,1+T\} .$\\
  Assume now that  $ c_{i}+d_{i}T$ is equal to T.\\
  Now $\begin{vmatrix}
    T & T \\  
    a_{j}+b_{j}T & c_{j}+d_{j}T 
\end{vmatrix} $ is equal to 0 for every $ j \neq i .$\\
We then conclude that  $   a_{j}+b_{j}T  = c_{j}+d_{j}T $ for every $ j \neq i,$ which contradicts the fact that 

$ \begin{pmatrix}
  a_{1}+b_{1}T\\  
    a_{2}+b_{2}T  \\
\vdots      \\
  a_{n}+b_{n}T 
  \end{pmatrix} \neq  \begin{pmatrix}
 c_{1}+d_{1}T \\  
   c_{2}+d_{2}T \\
\vdots    \\
  c_{n}+d_{n}T\\
  \end{pmatrix} $\\
  Hence we obtain that  $ c_{i}+d_{i}T \in \{1,1+T\} .$\\
 For instance let  $ c_{i}+d_{i}T $ be equal to 1+T, then 
   and as before $\begin{vmatrix}
    T & T+1 \\  
    a_{j}+b_{j}T & c_{j}+d_{j}T 
\end{vmatrix} $ is equal to 0 for every $ j \neq i.$\\
Thus $ ( a_{j}+b_{j}T , c_{j}+d_{j}T ) \in \{(0,0), (T,T+1)\} $ for every $ j \neq i.$ \\
Consider now the matrix M in the case $ a_{i}+b_{i}T = T$ and  $c_{i}+d_{i}T =1+T. $\\
Thus M is equal   $   \begin{pmatrix}
 a_{1}+b_{1}T &  c_{1}+d_{1}T               \\  
    a_{2}+b_{2}T &     c_{2}+d_{2}T     \\
\vdots  & \vdots    \\
T & 1+T \\  
  a_{i+1}+b_{i+1}T & c_{i+1}+d_{i+1}T \\
\vdots &   \vdots        \\
  a_{n}+b_{n}T & c_{n}+d_{n}T\\
  \end{pmatrix}, $\\ where $ ( a_{j}+b_{j}T , c_{j}+d_{j}T ) \in \{(0,0), (T,T+1)\} $ for every $ j \neq i.$ \\[0.2 cm]
  We have :
\begin{align}
  \label{eq 3.3}
  Card \bigg\{ M =   \begin{pmatrix}
  a_{1}+b_{1}T & c_{1}+d_{1}T \\  
    a_{2}+b_{2}T & c_{2}+d_{2}T \\
\vdots &   \vdots        \\
  a_{n}+b_{n}T & c_{n}+d_{n}T\\
  \end{pmatrix} \mid  ( a_{j}+b_{j}T , c_{j}+d_{j}T ) \in \{(0,0), (T,T+1)\}  \\
  for \;1\leqslant j \leqslant n \; and \;  M \neq (0) \bigg\} =  2^n-1\nonumber
\end{align}
  Obviously we get the same result if  in the formula \eqref{eq 3.3}  we replace the couple (T,T+1) by one of the following  couples (T+1,T), (1,T+1), (T+1,1), (1,T), (T,1)\\
   From the results above  we deduce that the number of matrices M of the form \eqref{eq 3.2} is equal to $ 6\cdot (2^n-1)$ \\
    Combining the above results  we get :\\
        \begin{equation}
        \label{eq 3.4}
    r_{i}  =  \begin{cases}
1 & \text{if  } i = 0,        \\
 3\cdot (2^{2n}-1) + 6\cdot (2^n-1)   & \text{if   } i=1,\\
  2^{4n} -3 \cdot 2^{2n}-6\cdot2^{n}+8   & \text{if   }  i=2 
    \end{cases}    
   \end{equation}
     \section{Computation of $R_{2}^{(k)}$}
  \label{sec 4}  
\begin{thm} 
\label{thm 4.1}
Let   $ R_{2}^{(k)} $ denote the number  of solutions \\
  $(Y_1,U_{1}^{(1)},U_{2}^{(1)}, \ldots,U_{n}^{(1)}, Y_2,U_{1}^{(2)},U_{2}^{(2)}, 
\ldots,U_{n}^{(2)})\in( \mathbb{F}_{2}[T])^{2n+2} ,$ \vspace{0.5 cm}\\
  of the polynomial equations  \vspace{0.2 cm}
 \[\left\{\begin{array}{c}
 Y_{1}U_{1}^{(1)} + Y_{2}U_{1}^{(2)} = 0  \\
    Y_{1}U_{2}^{(1)} + Y_{2}U_{2}^{(2)}  = 0\\
    \vdots \\
   Y_{1}U_{n}^{(1)} + Y_{2}U_{n}^{(2)}  = 0    \end{array}\right.\]
   $ \Leftrightarrow
    \begin{pmatrix}
   U_{1}^{(1)} & U_{1}^{(2)}  \\ 
      U_{2}^{(1)} & U_{2}^{(2)}   \\
\vdots &   \vdots    \\
U_{n}^{(1)} & U_{n}^{(2)}  \\
 \end{pmatrix}  \begin{pmatrix}
   Y_{1} \\
   Y_{2}
    \end{pmatrix} =   \begin{pmatrix}
  0 \\
  0 
   \end{pmatrix} $\\
    satisfying the degree conditions \\
                   $$  degY_i \leq k-1 ,
                   \quad degU_{j}^{(i)} \leq 1, \quad  for \quad 1\leq j\leq n  \quad 1\leq i \leq 2 $$ 
                   Then   $ R_{2}^{(k)} $   is equal to    \\
                    \begin{equation}
  \label{eq 4.1}
 \boxed{
  2^{2k}+ 3\cdot (2^{2n}-1)\cdot 2^k+6\cdot (2^{n}-1)\cdot 2^{k-1} + 2^{4n} -3 \cdot 2^{2n}-6\cdot2^{n}+8 }
\end{equation}
\end{thm}
\begin{proof}
\textbf{The case rank $  \begin{pmatrix}
   U_{1}^{(1)} & U_{1}^{(2)}\\    U_{2}^{(1)} & U_{2}^{(2)}\\
\vdots &   \vdots        \\
U_{n}^{(1)} & U_{n}^{(2)}\\
 \end{pmatrix} =0$}

  The number  of solutions  of the polynomial equations \\
 $   \begin{pmatrix}
  0 & 0\\  
  0 & 0\\
\vdots &   \vdots        \\
0& 0\\
 \end{pmatrix}  \begin{pmatrix}
   Y_{1} \\
   Y_{2}
  \end{pmatrix} =   \begin{pmatrix}
  0 \\
  0 \\
  \vdots \\
  0 
  \end{pmatrix} $\\

 satisfying the degree conditions \\
                   $$  degY_i \leq k-1 ,
                   \quad deg U_{j}^{(i)} = -\infty , \quad  for \quad 1\leq j\leq n  \quad 1\leq i \leq 2 $$ \\
  is obviously  equal to

  \begin{equation}
  \label{eq 4.2}
\boxed{2^k \cdot 2^k = 2^{2k}}
\end{equation}
  
\textbf{  The case rank $  \begin{pmatrix}
0   & U_{1}^{(2)}\\ 
     0 & U_{2}^{(2)}\\
\vdots &   \vdots        \\
0 & U_{n}^{(2)}\\
 \end{pmatrix} =1$}

  The number  of solutions  of the polynomial equations \\
                  
   $    \begin{pmatrix}
  0 & c_{1}+d_{1}T \\  
    0 & c_{2}+d_{2}T \\
\vdots &   \vdots        \\
  0 & c_{n}+d_{n}T\\
  \end{pmatrix}  \begin{pmatrix}
   Y_{1} \\
   Y_{2}
  \end{pmatrix} =   \begin{pmatrix}
  0 \\
  0 \\
  \vdots \\
  0 
  \end{pmatrix} $\\
  
  where  $ \begin{pmatrix}
  c_{1}+d_{1}T\\  
    c_{2}+d_{2}T  \\
\vdots      \\
  c_{n}+d_{n}T 
  \end{pmatrix} \neq  \begin{pmatrix}
0 \\  
   0 \\
\vdots    \\
  0\\
  \end{pmatrix} $\\

 satisfying the degree conditions \\
                   $$  degY_i \leq k-1 ,
                   \quad deg U_{j}^{(1)} = -\infty , \quad  deg U_{j}^{(2)} \leq 1 \quad  for \quad 1\leq j\leq n  \quad 1\leq i \leq 2 $$ \\

 is equal to 
   \begin{equation}
  \label{eq 4.3}
\boxed{(2^{2n}-1)\cdot 2^k }
\end{equation}

 \textbf{The case rank $  \begin{pmatrix}
   U_{1}^{(1)} & 0\\
       U_{2}^{(1)} & 0 \\
\vdots &   \vdots        \\
U_{n}^{(1)} &   0  \\
 \end{pmatrix} = 1$}

   The number  of solutions  of the polynomial equations \\
                  
   $    \begin{pmatrix}
 a_{1}+b_{1}T  &0 \\  
 a_{2}+b_{2}T     & 0\\
\vdots &   \vdots        \\
   a_{n}+b_{n}T & 0\\
  \end{pmatrix}  \begin{pmatrix}
   Y_{1} \\
   Y_{2}
  \end{pmatrix} =   \begin{pmatrix}
  0 \\
  0 \\
  \vdots \\
  0 
  \end{pmatrix} $\\
  
  where  $ \begin{pmatrix}
  a_{1}+b_{1}T\\  
    a_{2}+b_{2}T  \\
\vdots      \\
  a_{n}+b_{n}T 
  \end{pmatrix} \neq  \begin{pmatrix}
0 \\  
   0 \\
\vdots    \\
  0\\
  \end{pmatrix} $\\

 satisfying the degree conditions \\
                   $$  degY_i \leq 3 ,
                   \quad deg U_{j}^{(2)} = -\infty , \quad  deg U_{j}^{(1)} \leq 1 \quad  for \quad 1\leq j\leq n  \quad 1\leq i \leq 2 $$ \\

 is equal to 
    \begin{equation}
  \label{eq 4.4}
\boxed{(2^{2n}-1)\cdot 2^k }
\end{equation}

\textbf{The case rank $  \begin{pmatrix}
   U_{1}^{(1)} & U_{1}^{(1)}\\    U_{2}^{(1)} & U_{2}^{(1)}\\
\vdots &   \vdots        \\
U_{n}^{(1)} & U_{n}^{(1)}\\
 \end{pmatrix} = 1$}

   The number  of solutions  of the polynomial equations \\
                  
   $    \begin{pmatrix}
 a_{1}+b_{1}T  &  a_{1}+b_{1}T\\  
 a_{2}+b_{2}T     &   a_{2}+b_{2}T                 \\
\vdots &   \vdots        \\
   a_{n}+b_{n}T &  a_{n}+b_{n}T\\
  \end{pmatrix}  \begin{pmatrix}
   Y_{1} \\
   Y_{2}
  \end{pmatrix} =   \begin{pmatrix}
  0 \\
  0 \\
  \vdots \\
  0 
  \end{pmatrix} $\\
  
  where  $ \begin{pmatrix}
  a_{1}+b_{1}T\\  
    a_{2}+b_{2}T  \\
\vdots      \\
  a_{n}+b_{n}T 
  \end{pmatrix} \neq  \begin{pmatrix}
0 \\  
   0 \\
\vdots    \\
  0\\
  \end{pmatrix} $\\

 satisfying the degree conditions \\
                   $$  degY_i \leq k-1 ,
                   \quad  deg U_{j}^{(1)} \leq 1 \quad  for \quad 1\leq j\leq n  \quad 1\leq i \leq 2 $$ \\

 is equal to 
     \begin{equation}
  \label{eq 4.5}
\boxed{(2^{2n}-1)\cdot 2^k }
\end{equation}

\begin{equation}
 \label{eq 4.6}
  \textbf{The case rank}  
  \begin{pmatrix}
   U_{1}^{(1)} & U_{1}^{(2)}\\ 
      U_{2}^{(1)} & U_{2}^{(2)}\\
\vdots &   \vdots        \\
U_{n}^{(1)} & U_{n}^{(2)}\\
 \end{pmatrix} = 1
\end{equation}\\
\quad \text{where} $\quad  (U_{j}^{(1)},U_{j}^{2)}) \in \{(0,0),(T,T+1),(T+1,T),(1,T+1),(T+1,1),(T,1),(1,T)\}\text{for} \quad j \in [1,n] $
 See \eqref{eq 3.2}  and  \eqref{eq 3.3} refer to  Section \ref{sec 3}.\\
  Consider for instance the following case:\\
   The number  of solutions  of the polynomial equations \\
                    $    \begin{pmatrix}
 U_{1}^{(1)} & U_{1}^{(2)}\\ 
       \star    &   \star                \\ 
  \vdots &   \vdots        \\
    \star    &   \star                \\ 
  T  &  T+1 \\ 
   \star    &   \star                \\
\vdots &   \vdots        \\
 U_{n}^{(1)} & U_{n}^{(2)}   \\
  \end{pmatrix}  \begin{pmatrix}
   Y_{1} \\
   Y_{2}
  \end{pmatrix} =   \begin{pmatrix}
  0 \\
  0 \\
  \vdots \\
  0 
  \end{pmatrix} $\\
  satisfying the  conditions \\
                   $$  degY_i \leq k-1 ,
                   \quad (U_{j}^{(1)},U_{j}^{2)}) \in \{(0,0),(T,T+1)\}  \quad  for \quad 1\leq j\leq n,  \quad 1\leq i \leq 2 $$ is equal to  \\
          \begin{equation}
  \label{eq 4.7}
\boxed{ (2^{n}-1)\cdot 2^{k-1} }
\end{equation}
 Indeed the number  of solutions $(Y_{1},Y_{2}) \in ( \mathbb{F}_{2}[T] )^2 $ of the polynomial equation 
$ TY_{1}+(T+1)Y_{2} = 0 \quad \text{where}\quad  degY_i \leq k-1,\; i\in [1,2] $ is equal to $2^{k-1} $.\\
Apparently\\
\begin{align*}
 TY_{1}+(T+1)Y_{2} = 0 \Rightarrow T\mid Y_{2}, \quad T+1 \mid Y_{1} \\
  \Rightarrow   Y_{1} = (T+1)Y_{1}^{\star}, \quad  degY_{1}^{\star} \leq k-2 \quad
   \text{and} \quad  Y_{2} = TY_{2}^{\star}, \quad  degY_{2}^{\star} \leq k-2 \\
    \Rightarrow  T(T+1) Y_{1}^{\star}= T(T+1) Y_{2}^{\star} 
       \Rightarrow  Y_{1}^{\star}=Y_{2}^{\star} \quad  degY_{1}^{\star} \leq k-2 
\end{align*}
  Obviously if  we replace the couple (T,T+1) by one of the following five  couples (T+1,T), (1,T+1), (T+1,1), (1,T), (T,1)
  we obtain equally the formula  \eqref{eq 4.7} \\
  Thus the number of polynomial solutions in the case \eqref{eq 4.6} is equal to   
     \begin{equation}
  \label{eq 4.8}
\boxed{6\cdot (2^{n}-1)\cdot 2^{k-1} }
\end{equation}

\begin{equation}
 \label{eq 4.9}
  \textbf{The case rank}  
  \begin{pmatrix}
   U_{1}^{(1)} & U_{1}^{(2)}\\ 
      U_{2}^{(1)} & U_{2}^{(2)}\\
\vdots &   \vdots        \\
U_{n}^{(1)} & U_{n}^{(2)}\\
 \end{pmatrix} = 2
\end{equation}

  Thus using \eqref{eq 3.4} the number of polynomial solutions  \\ $  \begin{pmatrix}
   U_{1}^{(1)} & U_{1}^{(2)}  \\ 
      U_{2}^{(1)} & U_{2}^{(2)}   \\
\vdots &   \vdots    \\
U_{n}^{(1)} & U_{n}^{(2)}  \\
 \end{pmatrix}  \begin{pmatrix}
   Y_{1} \\
   Y_{2}
    \end{pmatrix} =   \begin{pmatrix}
  0 \\
  0 
   \end{pmatrix} $\\
in the case \eqref{eq 4.9} is equal to

     \begin{equation}
  \label{eq 4.10}
\boxed{  2^{4n} -3 \cdot 2^{2n}-6\cdot2^{n}+8 }
\end{equation}

So combining \eqref{eq 4.2}, \eqref{eq 4.3}, \eqref{eq 4.4} \eqref{eq 4.5} , \eqref{eq 4.8} and  \eqref{eq 4.10}  we obtain  \eqref{eq 4.1} in theorem 
 \ref{thm 4.1}

\end{proof}

\begin {example}

\textbf{Computation of  $ R_{2}^{(k)} $  in the case k=1}\\[0.1 cm]

  $ R_{2}^{(1)} $ is equal to the number of solutions $(a_{1},b_{1},c_{1},d_{1}, \ldots, a_{n},b_{n},c_{n},d_{n},\alpha, \beta )\in  \mathbb{F}_{2}^{4n+2} $ of the linear system :\\[0.1 cm]
 $   \begin{pmatrix}
   a_{1} & c_{1} \\ 
    b_{1}& d_{1} \\
\vdots &   \vdots   \\
    a_{n} & c_{n} \\ 
    b_{n}& d_{n} \\                                 
 \end{pmatrix} 
  \begin{pmatrix}
   \alpha\\
   \beta
  \end{pmatrix} =   \begin{pmatrix}
  0 \\
  0 \\
  \vdots \\
  0 
  \end{pmatrix} $\\
From \eqref{eq 4.1} with k=1 we obtain  $ R_{2}^{(1)} = 2^{4n}+3\cdot 2^{2n}$\\ 
Of course we can easily compute directly the number of solutions of the above linear system:\\
Indeed let M denote the above $2n\times 2$ matrix over $\mathbb{F}_{2} $ then (see  Fisher and Alexander [2] or Landsberg [1] )
  \begin{equation*}
       Card\{M \mid rank M=i\} =  \begin{cases}
1 & \text{if  } i = 0,        \\
   (2^{2n}-1)\cdot 3 & \text{if   } i=1,\\
2^{4n} - 3\cdot 2^{2n} +2 & \text{if   }  i=2 
    \end{cases}    
   \end{equation*}
Thus $ R_{2}^{(1)} = 4+ 2\cdot(2^{2n}-1)\cdot 3 +2^{4n} - 3\cdot 2^{2n} +2= 2^{4n}+3\cdot 2^{2n}$

\end{example}

\begin{example}
Combining \eqref{eq 2.1} with q=2 and \eqref{eq 4.1} we obtain for $ k\leqslant 2n $ :\\

  \begin{align}
  \label{eq 4.11}
 R_{2}^{(k)} =  2^{2(2n+k) - (k+1)n}\sum_{i = 0}^{k}
  \Gamma_{i}^{\left[2\atop{\vdots \atop 2}\right]\times k} 2^{-2i}  \\ 
   =  2^{2k}+ 3\cdot (2^{2n}-1)\cdot 2^k+6\cdot (2^{n}-1)\cdot 2^{k-1} + 2^{4n} -3 \cdot 2^{2n}-6\cdot2^{n}+8 \nonumber
 \end{align}

\end{example}
\newpage

   \section{Computation of the number  $ \Gamma_{i}^{\left[2\atop{\vdots \atop 2}\right]\times k} $   of   n-times persymmetric 
   $2n  \times k$  rank i matrices  for $1\leqslant k \leqslant 6 $}
  \label{sec 5}  
  
We observe that the $ \Gamma_{i}^{\left[2\atop{\vdots \atop 2}\right]\times k} $   where  $ 0\leq i\leq k$ (see Section \ref{sec 2})
are solutions to the system 
\begin{lem}
\label{lem 5.1}
 \begin{equation}
  \label{eq 5.1}
 \begin{cases} 
 \displaystyle  \Gamma_{0}^{\left[2\atop{\vdots \atop 2}\right]\times k}  = 1 \quad \text{if} \quad  k\geqslant 1 \\
\displaystyle  \Gamma_{1}^{\left[2\atop{\vdots \atop 2}\right]\times k}  = (2^{n}-1)\cdot 3 \quad \text{if} \quad  k\geqslant 2 \\
 \displaystyle  \sum_{i = 0}^{k} \Gamma_{i}^{\left[2\atop{\vdots \atop 2}\right]\times k}  = 2^{(k+1)n} \\ 
  \displaystyle  \sum_{i = 0}^{k} \Gamma_{i}^{\left[2\atop{\vdots \atop 2}\right]\times k} 2^{-i}  = 2^{n+k(n-1)}+2^{(k-1)n}-2^{(k-1)n-k}\\
  \displaystyle \sum_{i = 0}^{k} \Gamma_{i}^{\left[2\atop{\vdots \atop 2}\right]\times k} 2^{-2i}  =
   2^{n+k(n-2)}+2^{-n+k(n-2)}\cdot[3\cdot2^k-3] +2^{-2n+k(n-2)}\cdot[6\cdot2^{k-1}-6] \\
   +2^{-3n+kn}-6\cdot2^{n(k-3)-k}+8\cdot2^{-3n+k(n-2)}
\end{cases}
    \end{equation}
  \end{lem}  
\begin{proof}
Lemma \ref{lem 5.1} follows from \eqref{eq 2.5},\eqref{eq 2.4}, \eqref{eq 2.3} and \eqref{eq 4.11}.
\end{proof}

  \begin{lem}
\label{lem 5.2}
\textbf{The case k=1}
  \begin{equation*}
       \Gamma_{i}^{\left[2\atop{\vdots \atop 2}\right]\times 1} =  \begin{cases}
1 & \text{if  } i = 0,        \\
   2^{2n} -1& \text{if   } i=1.
 \end{cases}    
   \end{equation*}

\textbf{The case k=2}

  \begin{equation*}
       \Gamma_{i}^{\left[2\atop{\vdots \atop 2}\right]\times 2} =  \begin{cases}
1 & \text{if  } i = 0,        \\
   (2^{n}-1)\cdot 3 & \text{if   } i=1,\\
2^{  3n} - 3\cdot 2^{n} +2 & \text{if   }  i=2. 
    \end{cases}    
   \end{equation*}

\textbf{The case k=3}

    \begin{equation*}
       \Gamma_{i}^{\left[2\atop{\vdots \atop 2}\right]\times 3} =  \begin{cases}
1 & \text{if  } i = 0,        \\
   (2^{n}-1)\cdot 3 & \text{if   } i=1,\\
  7\cdot 2^{2n} -9\cdot2^{n}+2 & \text{if   }  i = 2,  \\
2^{  4n} - 7\cdot 2^{2n} +6\cdot2^{n}  & \text{if   }  i=3. 
    \end{cases}    
   \end{equation*}

\textbf{The case k=4}

       \begin{equation*}
      \Gamma_{i}^{\left[2\atop{\vdots \atop 2}\right]\times 4}=   \begin{cases}
1 & \text{if  } i = 0,        \\
 (2^{n}-1)\cdot 3 & \text{if   } i=1,\\
7\cdot2^{2n}+7\cdot2^{n}-14 & \text{if   }  i = 2,  \\
 15\cdot 2^{3n}-21\cdot 2^{2n}-42\cdot 2^{n}+48   & \text{if   }  i = 3,  \\
 2^{5n} -15\cdot 2^{3n}+7\cdot 2^{2n+1}+2^{n+5}-32 & \text{if   }  i=4. 
    \end{cases}    
   \end{equation*}
 \end{lem}  
  \begin{proof}
  Lemma \ref{lem 5.2} follows from Lemma \ref{lem 5.1}.
  
  \end{proof}
We shall need the following Lemma.\\
  \begin{lem}
\label{lem 5.3}
\begin{equation}
\label{eq 5.2}
   \Gamma_{2}^{\left[2\atop{\vdots \atop 2}\right]\times k}=   \begin{cases}
2^{k+1}-4 & \text{if  } n = 1,        \\
 3\cdot 2^{k+1}+30 & \text{if   } n=2,\\
 7\cdot 2^{k+1}+266 & \text{if   }  n = 3. 
    \end{cases}
   \end{equation} 
    
 \begin{equation}
   \label{eq 5.3}    
    \Gamma_{3}^{\left[2\atop{\vdots \atop 2}\right]\times k}=   \begin{cases}
0 & \text{if  } n = 1,        \\
 21\cdot 2^{k+1}-168 & \text{if   } n=2,\\
147\cdot 2^{k+1}+1344 & \text{if   }  n = 3. 
    \end{cases}    
  \end{equation}
      \end{lem}
  \begin{proof}
  \end{proof}
  
  Lemma \ref{lem 5.3} follows from Daykin [3], Cherly [4],[5]  and  [6].

       \begin{lem}
\label{lem 5.4}
We postulate that :\\
\begin{equation}
\label{eq 5.4}
 \Gamma_{2}^{\left[2\atop{\vdots \atop 2}\right]\times k} = 7\cdot2^{2n}+(2^{k+1}-25) \cdot 2^{n}-2^{k+1}+18 \quad \text{for} \quad k\geqslant 3
      \end{equation}

\end{lem}
  \begin{proof}
  From the expressions of $ \Gamma_{2}^{\left[2\atop{\vdots \atop 2}\right]\times k} $ for k=3 and k=4 in Lemma \ref{lem 5.2} we assume that 
  $ \Gamma_{2}^{\left[2\atop{\vdots \atop 2}\right]\times k} $ can be written in the form :\\
  \begin{equation}
\label{eq 5.5}
 7\cdot2^{2n}+a(k)\cdot2^{n}+b(k) 
\end{equation}
Combining \eqref{eq 5.5} and \eqref{eq 5.2} for n=1,n=2 and n=3 we obtain : \\
 \begin{equation}
  \label{eq 5.6}
 \begin{cases} 
 \displaystyle 2a(k)+b(k)=2^{k+1}-32 \\
 \displaystyle  4a(k)+b(k)=3\cdot 2^{k+1}-82 \\
 \displaystyle    8a(k)+b(k)=7\cdot 2^{k+1}-182 \\
 \end{cases}
    \end{equation}
From \eqref{eq 5.6} we deduce :\\
$a(k)=2^{k+1}-25 $ and  $b(k)=-2^{k+1}+18.$
  \end{proof}
  
  \begin{lem}
  \label{lem 5.5}
We postulate that :\\
\begin{equation}
\label{eq 5.7}
 \Gamma_{3}^{\left[2\atop{\vdots \atop 2}\right]\times k} = 15\cdot2^{3n} + (7\cdot2^k-133)\cdot2^{2n}+ (294-21\cdot 2^k) \cdot 2^{n}   -176+14\cdot2^k \quad \text{for} \quad k\geqslant 4
      \end{equation}
\end{lem}
  \begin{proof}
  From the expression of $ \Gamma_{3}^{\left[2\atop{\vdots \atop 2}\right]\times k} $ for k=4 in Lemma \ref{lem 5.2} we assume that 
  $ \Gamma_{3}^{\left[2\atop{\vdots \atop 2}\right]\times k} $ can be written in the form :\\
  \begin{equation}
\label{eq 5.8}
15\cdot2^{3n} + a(k)\cdot2^{2n}+b(k)\cdot2^{n}+c(k) 
\end{equation}
Combining \eqref{eq 5.8} and \eqref{eq 5.3} for n=1,n=2 and n=3 we obtain : \\
 \begin{equation}
  \label{eq 5.9}
 \begin{cases} 
 \displaystyle 4a(k)+2b(k)+c(k) = -120 \\
 \displaystyle  16a(k)+4b(k)+c(k) =21\cdot 2^{k+1}-1128 \\
  \displaystyle    64a(k)+8b(k)+c(k) =147\cdot 2^{k+1}-6336 \\
 \end{cases}
    \end{equation}
From \eqref{eq 5.9} we deduce :\\
$a(k)=7\cdot2^k -133,\; b(k)=294-21\cdot2^k$ and $c(k)= -176+14\cdot2^k$
  \end{proof}

    \begin{lem}
\label{lem 5.6}
\textbf{The case k=5}
  \begin{equation}
  \label{eq 5.10}
     \Gamma_{i}^{\left[2\atop{\vdots \atop 2}\right]\times 5}=   \begin{cases}
1 & \text{if  } i = 0,        \\
 (2^{n}-1)\cdot 3 & \text{if   } i=1,\\
7\cdot2^{2n}+39 \cdot 2^{n}-46  & \text{if   }  i = 2,  \\
 15\cdot 2^{3n}+91\cdot 2^{2n}-189\cdot 2^{n+1}+272   & \text{if   }  i = 3,  \\
31\cdot 2^{4n} -45\cdot 2^{3n}-161\cdot 2^{2n+1}+51\cdot2^{n+4}-480 & \text{if   }  i=4, \\                                                                                                
2^{6n} - 31\cdot 2^{4n} +15\cdot 2^{3n+1}+7\cdot 2^{2n+5}-15\cdot2^{n+5}+256   & \text{if   }  i=5. 
  \end{cases}    
 \end{equation}

\textbf{The case k=6}

  \begin{equation}
  \label{eq 5.11}
    \Gamma_{i}^{\left[2\atop{\vdots \atop 2}\right]\times 6}=   \begin{cases}
1 & \text{if  } i = 0,        \\
 (2^{n}-1)\cdot 3 & \text{if   } i=1,\\
7\cdot2^{2n}+103 \cdot 2^{n}-110  & \text{if   }  i = 2,  \\
 15\cdot 2^{3n}+315\cdot 2^{2n}-1050\cdot 2^{n}+720   & \text{if   }  i = 3,  \\
31\cdot 2^{4n} +515\cdot 2^{3n}-2450 \cdot 2^{2n}+3280 \cdot2^{n}-1376  & \text{if   }  i=4, \\                                                                                                
   63 \cdot 2^{5n}-93 \cdot2^{4n} -1650 \cdot2^{3n}+5040 \cdot 2^{2n}-4128 \cdot 2^n +768 & \text{if   }  i=5, \\
      2^{7n}-63 \cdot 2^{5n}+62\cdot 2^{4n}+1120\cdot 2^{3n}-2912 \cdot 2^{2n}+1792 \cdot2^{n}   & \text{if   }  i=6. 
  \end{cases}    
  \end{equation}

   \end{lem} 
\begin{proof}
  Lemma \ref{lem 5.6} follows from Lemma \ref{lem 5.1},  \eqref{eq 5.4} and \eqref{eq 5.7}
  \end{proof}
 \begin{example} 
   We recover from \eqref{eq 5.11} for n=3 the formula given in Cherly [7] for the number $ \Gamma_{i}^{\left[2\atop{2 \atop2}\right]\times 6}$ of the  rank i triple persymmetric $6\times 6$ matrix over $\mathbb{F}_{2}$  of the below form
    \begin{displaymath}
   \left (  \begin{array} {cccccc}
\alpha  _{1}^{(1)} & \alpha  _{2}^{(1)}  &   \alpha_{3}^{(1)} &   \alpha_{4}^{(1)} &   \alpha_{5}^{(1)}  &   \alpha_{6}^{(1)} \\
\alpha  _{2}^{(1)} & \alpha  _{3}^{(1)}  &   \alpha_{4}^{(1)} &   \alpha_{5}^{(1)} &   \alpha_{6}^{(1)}  &   \alpha_{7}^{(1)} \\   
 \hline \\
\alpha  _{1}^{(2)} & \alpha  _{2}^{(2)}  &   \alpha_{3}^{(2)} &   \alpha_{4}^{(2)} &   \alpha_{5}^{(2)}  &   \alpha_{6}^{(2)} \\   
 \alpha  _{2}^{(2)} & \alpha  _{3}^{(2)}  &   \alpha_{4}^{(2)} &   \alpha_{5}^{(2)}&   \alpha_{6}^{(2)}  &   \alpha_{7}^{(2)} \\   
  \hline\\
\alpha  _{1}^{(3)} & \alpha  _{2}^{(3)}  &   \alpha_{3}^{(3)}  &   \alpha_{4}^{(3)} &   \alpha_{5}^{(3)} &   \alpha_{6}^{(3)} \\   
 \alpha  _{2}^{(3)} & \alpha  _{3}^{(3)}  &   \alpha_{4}^{(3)}&   \alpha_{5}^{(3)} &   \alpha_{6}^{(3)}  &   \alpha_{7}^{(3)}    
  \end{array} \right )  
\end{displaymath} 
that is :\\
\[ \Gamma_{i}^{\left[2\atop{ 2\atop 2} \right]\times k}=
 \begin{cases}
1  &\text{if  }  i = 0 \\
21   &\text{if  }  i = 1 \\
1162  & \text{if   } i = 2 \\
20160   & \text{if   } i = 3 \\
258720 & \text{if   } i = 4 \\
1128960   & \text{if   } i = 5 \\
688128  & \text{if  } i = 6.
\end{cases}
\]

  \end{example}
  
  For other related articles concerning persymmetric matrices over the finite field with two elements see Cherly    [8]

  \end{document}